\documentclass[12pt]{amsart}

\usepackage[utf8]{inputenc}
\usepackage{etex}
\usepackage{tabularx}
\usepackage{amsmath, amssymb, amsthm, color}
\usepackage{tikz}
\usepackage{xifthen, graphicx}
\usepackage{enumerate}
\usepackage{graphicx}%
\usepackage{multirow}%
\usepackage{amsmath,amssymb,amsfonts}%
\usepackage{amsthm}%
\usepackage{mathrsfs}%
\usepackage[title]{appendix}%
\usepackage{xcolor}%
\usepackage{textcomp}%
\usepackage{manyfoot}%
\usepackage{booktabs}%
\usepackage{algorithm}%
\usepackage{algorithmicx}%
\usepackage{algpseudocode}%
\usepackage{listings}%
\usepackage{tikz}
\usepackage{bbm}

\theoremstyle{definition}
\newtheorem{thm}{Theorem}

\newtheorem{ex}{Example}

\newtheorem{prop}[thm]{Proposition}

\title{On the Coxeter Cohomology of Irreducible Representations of Symmetric Groups}
\author{Hayley Bertrand*}
\address{University of Wisconsin - Whitewater}
\email{bertranh@uww.edu}
\author{Jing Anne McLaughlin}

\begin{document}

\begin{abstract}
This work is part of a research program to compute the Hochschild homology groups $\text{HH}_*(\mathbb{C}[x_1, ..., x_d]/(x_1, ..., x_d)^3; \mathbb{C})$ via Coxeter cohomology, which utilizes the isomorphism

\[ \text{HH}_i(\mathbb{C}[x_1, ..., x_d]/(x_1, ..., x_d)^3; \mathbb{C}) \cong \sum_{0 \leq j \leq i} H_C^j\bigg( S_{i+j}, V^{\otimes(i+j)} \bigg)\]

\noindent provided by Larsen and Lindenstrauss. Here, $H^*_C$ denotes Coxeter cohomology, $S_{i+j}$ is the symmetric group on $i+j$ letters, and $V$ is the standard representation of $\text{GL}_d(\mathbb{C})$ on $\mathbb{C}^d$. While previous work has focused exclusively on the case $d=2$, we extend some results to all values of $d$. Notably, we show that when the tensor representation is replaced by an irreducible representation, the Euler characteristic of the corresponding Coxeter cohomology is given by a polynomial with bounded degree. Although the problem is motivated by algebra and topology, the solution relies primarily on combinatorial arguments.

\end{abstract}

\maketitle

\section{Introduction}
\label{intro}
The motivation for this work is to make progress toward the computation of the Hochschild homology of a certain family of quotients of polynomial rings; Hochschild homology, a well-studied invariant for rings, is of interest to both algebraists and algebraic topologists. More specifically, we would like to compute the Hochschild homology of rings of the form $k[x_1, ..., x_d]/(x_1, ..., x_d)^3$, especially in the case where $k = \mathbb{C}$. In \cite{L}, Lindenstrauss establishes an isomorphism which relates the Hochschild homology HH$_*(k[x_1, ..., x_d]/(x_1, ..., x_d)^a)$ for any field $k$ of characteristic 0 to the relative homology of tori modulo fat diagonals, and uses this result to compute the desired Hochschild homology in the case $a=2$. For $a \geq 3$, however, the same techniques cannot be applied due to the complexity of the combinatorics involved. 

Instead, in \cite{LL}, Larsen and Lindenstrauss lay the groundwork for \textit{Coxeter cohomology} and establish an isomorphism which involves a sum of Coxeter cohomology groups:

\[ \text{HH}_i(\mathbb{C}[x_1, .., x_d]/(x_1, ..., x_d)^3; \mathbb{C)} \cong \bigoplus_{0\leq j \leq i} H^j_C (S_{i+j}, V^{\otimes (i+j)})\]

Here $H_C^*$ denotes Coxeter cohomology, $S_{i+j}$ the symmetric group on $i+j$ letters, and $V$ the standard representation of GL$_d(\mathbb{C})$ on $\mathbb{C}^d$. This allows us to instead approach the problem via representation theory of symmetric groups.

In \cite{B}, Bertrand uses the results of \cite{LL} to compute the Euler characteristic of the relevant Coxeter cohomology groups, of the form $\chi(H_C^*(S_n, V^{\otimes n}))$ where $V$ is again the standard representation of GL$_d(\mathbb{C})$ on $\mathbb{C}^d$. With a goal of recovering the Coxeter cohomology, Bertrand then looks at Euler characteristics of the form $\chi(H_C^*(S_n, V))$ where $V$ is an irreducible representation of $S_n$. Bertrand makes use of the bijection between these irreducible representations and their associated Young Diagrams, showing that $d$ bounds the maximum number of rows of the Young diagrams which must be considered, and investigates the case $d=2$, where the corresponding Young diagrams have at most two rows. In the case where $V$ is irreducible, Bertrand finds the Euler characteristic to be given by a polynomial with certain properties:

\begin{thm}[\cite{B}, Thm. 11] \label{thm1}
Let $V_{(i,j)}$ be the irreducible representation of $S_n$ whose corresponding Young diagram is given by $\lambda = (i, j)$. There exist three polynomials, each of degree at most $j$, such that

\begin{align*}
 \chi(H_C^*(S_n, V_{(i, j)})) = 
\begin{cases}
(-1)^{i}p^j_0(i+j) & i+j \equiv 0 \text{ mod }3 \\
(-1)^{i}p^j_1(i+j) & i+j \equiv 1 \text{ mod }3 \\
(-1)^{i}p^j_2(i+j) & i+j \equiv 2 \text{ mod }3 .
 \end{cases} 
 \end{align*}

\end{thm}

The aim of this work, then, is to extend this result to $d > 2$. Unfortunately, the arguments from \cite{B} become increasingly more combinatorially complex as the number of rows in the Young diagram increases, so we instead take a different approach. Although the method differs, the result is analogous, giving a similar description for all Young diagrams:

\begin{thm}
Let $\lambda = (\lambda_1, \lambda_2, ..., \lambda_k)$ be a Young diagram of size $n$ where $\lambda_1 - \lambda_2 \geq 6$. Then there exist six polynomials, $p_0, p_1, p_2, p_3, p_4$, and $p_5$, each of degree at most $n - \lambda_1 = \sum_{i=2}^k \lambda_i$, such that \\ $\chi(H^*_C(S_n, V_{\lambda})) = p_i(n)$ where $n \equiv i \mod 6$.
\end{thm}

\section{Preliminaries}
\label{prelim}

We begin by establishing notation, most of which is standard. Although Coxeter cohomology $H^*_C(G,V)$ is defined for any pair $(G, V)$, where $G$ is a Coxeter group and $V$ is a representation of $G$, within this work, the Coxeter group in question will always be isomorphic to some symmetric group, and we will consider only complex representations.
Because our results describe the Euler characteristic of the Coxeter cohomology in question, we frequently condense notation by writing $\chi(G, V )$ (or $\chi(V)$ if $G$ is clear from context) to denote $\chi(H_C^* (G, V))$. We work heavily with irreducible representations of symmetric groups and their associated Young diagrams; if $\lambda = (\lambda_1, \lambda_2, ..., \lambda_k)$ is a Young diagram with $\lambda_i$ boxes in row $i$, then $V_{\lambda}$ is its corresponding irreducible representation. 

We also include here the relevant results from \cite{B}. The first establishes a recurrence between Euler characteristics. 

\begin{prop}[\cite{B}, Cor. 4]
For every complex representation $V$ of $S_n$,

\[\chi(S_n,V)=\chi(S_{n-1},V|_{S_{n-1}})-\chi(S_{n-2},V^{\langle(n-1\text{ }n)\rangle})\]

where $(n-1 \text{ } n) \in S_n$ is a transposition.

\end{prop}

In order to apply this recurrence in practice, we must first describe both terms on the right hand side in terms of Young diagrams. The following result does so in the case where $V$ is an irreducible representation of $S_n$.

\begin{prop}[\cite{B}, Lemma 9]

Let $\lambda$ be a Young diagram of size $n$. Then

\begin{enumerate}

\item $V_{\lambda}|_{S_{n-1}} = \bigoplus _{\mu} V_{\mu}$, where the sum is over all Young diagrams $\mu$ obtained from $\lambda$ by removing a single box

\item $V_{\lambda}^{\langle \text{($n-1$ $n$)} \rangle} = \bigoplus _{\mu} V_{\mu}$, where the sum is over all Young diagrams $\mu$ obtained from $\lambda$ by removing two boxes, not both from the same column.

\end{enumerate}

\end{prop}

We give two examples to illustrate how these two results can be applied in practice. 

\begin{ex}

Let $\lambda = (6,4,2)$. Then $V_{\lambda}$ is a representation of $S_{12}$, and by Prop. 3, $\chi(V_{\lambda}) = \chi(V_{\lambda}|_{S_{11}}) - \chi(V_{\lambda}^{\langle (11 \text{ } 12) \rangle})$. We decompose both resulting terms using Prop. 4, both in terms of diagrams and the corresponding Euler characteristics. In Fig. \ref{fig:1}, the first row of each corresponds to $\chi(V_{\lambda}|_{S_{11}})$, and the second and third rows correspond to $\chi(V_{\lambda}^{\langle (11 \text{ } 12) \rangle})$.

\end{ex}

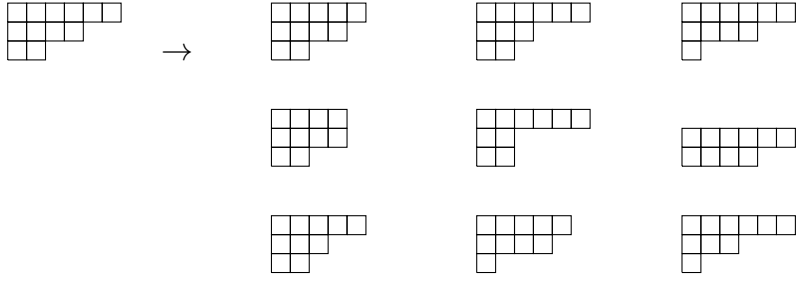
\begin{figure}
\begin{center}
\begin{tabular}{lcclclcl}
    
\begin{tikzpicture}
\def\n{6};
\foreach \x in {1,...,\n}{
\draw (.25*\x,0)--(.25*\x+.25,0)--(.25*\x+.25,.25)--(.25*\x,.25)--(.25*\x,0);
}
\def\nn{4};
\foreach \x in {1,...,\nn}{
\draw (.25*\x,-.25)--(.25*\x+.25,-.25)--(.25*\x+.25,0)--(.25*\x,0)--(.25*\x,-.25);
}
\def\nnn{2};
\foreach \x in {1,...,\nnn}{
\draw (.25*\x,-.5)--(.25*\x+.25,-.5)--(.25*\x+.25,-.25)--(.25*\x,-.25)--(.25*\x,-.5);
}
\end{tikzpicture} & $\rightarrow$ &&
\begin{tikzpicture}
\def\n{5};
\foreach \x in {1,...,\n}{
\draw (.25*\x,0)--(.25*\x+.25,0)--(.25*\x+.25,.25)--(.25*\x,.25)--(.25*\x,0);
}
\def\nn{4};
\foreach \x in {1,...,\nn}{
\draw (.25*\x,-.25)--(.25*\x+.25,-.25)--(.25*\x+.25,0)--(.25*\x,0)--(.25*\x,-.25);
}
\def\nnn{2};
\foreach \x in {1,...,\nnn}{
\draw (.25*\x,-.5)--(.25*\x+.25,-.5)--(.25*\x+.25,-.25)--(.25*\x,-.25)--(.25*\x,-.5);
}
\end{tikzpicture} &&
\begin{tikzpicture}
\def\n{6};
\foreach \x in {1,...,\n}{
\draw (.25*\x,0)--(.25*\x+.25,0)--(.25*\x+.25,.25)--(.25*\x,.25)--(.25*\x,0);
}
\def\nn{3};
\foreach \x in {1,...,\nn}{
\draw (.25*\x,-.25)--(.25*\x+.25,-.25)--(.25*\x+.25,0)--(.25*\x,0)--(.25*\x,-.25);
}
\def\nnn{2};
\foreach \x in {1,...,\nnn}{
\draw (.25*\x,-.5)--(.25*\x+.25,-.5)--(.25*\x+.25,-.25)--(.25*\x,-.25)--(.25*\x,-.5);
}
\end{tikzpicture} &&
\begin{tikzpicture}
\def\n{6};
\foreach \x in {1,...,\n}{
\draw (.25*\x,0)--(.25*\x+.25,0)--(.25*\x+.25,.25)--(.25*\x,.25)--(.25*\x,0);
}
\def\nn{4};
\foreach \x in {1,...,\nn}{
\draw (.25*\x,-.25)--(.25*\x+.25,-.25)--(.25*\x+.25,0)--(.25*\x,0)--(.25*\x,-.25);
}
\def\nnn{1};
\foreach \x in {1,...,\nnn}{
\draw (.25*\x,-.5)--(.25*\x+.25,-.5)--(.25*\x+.25,-.25)--(.25*\x,-.25)--(.25*\x,-.5);
}
\end{tikzpicture} \\ &&&&&&& \\
& & &
\begin{tikzpicture}
\def\n{4};
\foreach \x in {1,...,\n}{
\draw (.25*\x,0)--(.25*\x+.25,0)--(.25*\x+.25,.25)--(.25*\x,.25)--(.25*\x,0);
}
\def\nn{4};
\foreach \x in {1,...,\nn}{
\draw (.25*\x,-.25)--(.25*\x+.25,-.25)--(.25*\x+.25,0)--(.25*\x,0)--(.25*\x,-.25);
}
\def\nnn{2};
\foreach \x in {1,...,\nnn}{
\draw (.25*\x,-.5)--(.25*\x+.25,-.5)--(.25*\x+.25,-.25)--(.25*\x,-.25)--(.25*\x,-.5);
}
\end{tikzpicture} &&
\begin{tikzpicture}
\def\n{6};
\foreach \x in {1,...,\n}{
\draw (.25*\x,0)--(.25*\x+.25,0)--(.25*\x+.25,.25)--(.25*\x,.25)--(.25*\x,0);
}
\def\nn{2};
\foreach \x in {1,...,\nn}{
\draw (.25*\x,-.25)--(.25*\x+.25,-.25)--(.25*\x+.25,0)--(.25*\x,0)--(.25*\x,-.25);
}
\def\nnn{2};
\foreach \x in {1,...,\nnn}{
\draw (.25*\x,-.5)--(.25*\x+.25,-.5)--(.25*\x+.25,-.25)--(.25*\x,-.25)--(.25*\x,-.5);
}
\end{tikzpicture} &&
\begin{tikzpicture}
\def\n{6};
\foreach \x in {1,...,\n}{
\draw (.25*\x,0)--(.25*\x+.25,0)--(.25*\x+.25,.25)--(.25*\x,.25)--(.25*\x,0);
}
\def\nn{4};
\foreach \x in {1,...,\nn}{
\draw (.25*\x,-.25)--(.25*\x+.25,-.25)--(.25*\x+.25,0)--(.25*\x,0)--(.25*\x,-.25);
}
\end{tikzpicture} \\ &&&&&&& \\ &&&
\begin{tikzpicture}
\def\n{5};
\foreach \x in {1,...,\n}{
\draw (.25*\x,0)--(.25*\x+.25,0)--(.25*\x+.25,.25)--(.25*\x,.25)--(.25*\x,0);
}
\def\nn{3};
\foreach \x in {1,...,\nn}{
\draw (.25*\x,-.25)--(.25*\x+.25,-.25)--(.25*\x+.25,0)--(.25*\x,0)--(.25*\x,-.25);
}
\def\nnn{2};
\foreach \x in {1,...,\nnn}{
\draw (.25*\x,-.5)--(.25*\x+.25,-.5)--(.25*\x+.25,-.25)--(.25*\x,-.25)--(.25*\x,-.5);
}
\end{tikzpicture} &&
\begin{tikzpicture}
\def\n{5};
\foreach \x in {1,...,\n}{
\draw (.25*\x,0)--(.25*\x+.25,0)--(.25*\x+.25,.25)--(.25*\x,.25)--(.25*\x,0);
}
\def\nn{4};
\foreach \x in {1,...,\nn}{
\draw (.25*\x,-.25)--(.25*\x+.25,-.25)--(.25*\x+.25,0)--(.25*\x,0)--(.25*\x,-.25);
}
\def\nnn{1};
\foreach \x in {1,...,\nnn}{
\draw (.25*\x,-.5)--(.25*\x+.25,-.5)--(.25*\x+.25,-.25)--(.25*\x,-.25)--(.25*\x,-.5);
}
\end{tikzpicture} &&
\begin{tikzpicture}
\def\n{6};
\foreach \x in {1,...,\n}{
\draw (.25*\x,0)--(.25*\x+.25,0)--(.25*\x+.25,.25)--(.25*\x,.25)--(.25*\x,0);
}
\def\nn{3};
\foreach \x in {1,...,\nn}{
\draw (.25*\x,-.25)--(.25*\x+.25,-.25)--(.25*\x+.25,0)--(.25*\x,0)--(.25*\x,-.25);
}
\def\nnn{1};
\foreach \x in {1,...,\nnn}{
\draw (.25*\x,-.5)--(.25*\x+.25,-.5)--(.25*\x+.25,-.25)--(.25*\x,-.25)--(.25*\x,-.5);
}
\end{tikzpicture} \\ &&&&&&& \\
$\chi(V_{(6,4,2)})$& = && $\chi(V_{(5,4,2)})$ & + &$\chi(V_{(6,3,2)})$ & + & $\chi(V_{(6,4,1)})$ \\ &&&&&&& \\
&& $-$ & $\chi(V_{(4,4,2)})$ & $-$ & $\chi(V_{(6,2,2)})$ & $-$ & $\chi(V_{(6,4)})$ \\ &&&&&&& \\
&& $-$ & $\chi(V_{(5,3,2)})$ & $-$ & $\chi(V_{(5,4,1)})$ & $-$ & $\chi(V_{(6,3,1)})$ \\
\end{tabular}
\end{center}
\caption{Decomposition of $V_{(6, 4, 2)}$}
\label{fig:1}       
\end{figure}

When $\lambda$ has three rows, Prop. 4 gives up to nine possible terms in the decomposition, as shown in the previous example. As the number of rows of $\lambda$ increases, the number of possible terms increases quickly. However, the maximum number of terms is often not reached, as demonstrated in the following example.

\begin{ex}
Let $\lambda = (3,3,1)$. Then $V_{\lambda}$ is a representation of $S_{7}$, and by by Prop. 3, $\chi(V_{\lambda}) = \chi(V_{\lambda}|_{S_{6}}) - \chi(V_{\lambda}^{\langle (6 \text{ } 7) \rangle})$. We again decompose both resulting terms using Prop. 4, and again the first row of each portion of Fig. \ref{fig:2} corresponds to $\chi(V_{\lambda}|_{S_{6}})$,while the second and third rows correspond to $\chi(V_{\lambda}^{\langle (6 \text{ } 7) \rangle})$.

\end{ex}

\begin{figure}
\begin{center}
\begin{tabular}{lcclcl}
\begin{tikzpicture}
\def\n{3};
\foreach \x in {1,...,\n}{
\draw (.5*\x,0)--(.5*\x+.5,0)--(.5*\x+.5,.5)--(.5*\x,.5)--(.5*\x,0);
}
\def\nn{3};
\foreach \x in {1,...,\nn}{
\draw (.5*\x,-.5)--(.5*\x+.5,-.5)--(.5*\x+.5,0)--(.5*\x,0)--(.5*\x,-.5);
}
\def\nnn{1};
\foreach \x in {1,...,\nnn}{
\draw (.5*\x,-1)--(.5*\x+.5,-1)--(.5*\x+.5,-.5)--(.5*\x,-.5)--(.5*\x,-1);
}
\end{tikzpicture} & $\rightarrow$ &&
\begin{tikzpicture}
\def\n{3};
\foreach \x in {1,...,\n}{
\draw (.5*\x,0)--(.5*\x+.5,0)--(.5*\x+.5,.5)--(.5*\x,.5)--(.5*\x,0);
}
\def\nn{2};
\foreach \x in {1,...,\nn}{
\draw (.5*\x,-.5)--(.5*\x+.5,-.5)--(.5*\x+.5,0)--(.5*\x,0)--(.5*\x,-.5);
}
\def\nnn{1};
\foreach \x in {1,...,\nnn}{
\draw (.5*\x,-1)--(.5*\x+.5,-1)--(.5*\x+.5,-.5)--(.5*\x,-.5)--(.5*\x,-1);
}
\end{tikzpicture} &&
\begin{tikzpicture}
\def\n{3};
\foreach \x in {1,...,\n}{
\draw (.5*\x,0)--(.5*\x+.5,0)--(.5*\x+.5,.5)--(.5*\x,.5)--(.5*\x,0);
}
\def\nn{3};
\foreach \x in {1,...,\nn}{
\draw (.5*\x,-.5)--(.5*\x+.5,-.5)--(.5*\x+.5,0)--(.5*\x,0)--(.5*\x,-.5);
}
\end{tikzpicture} \\ &&&&& \\ &&&
\begin{tikzpicture}
\def\n{3};
\foreach \x in {1,...,\n}{
\draw (.5*\x,0)--(.5*\x+.5,0)--(.5*\x+.5,.5)--(.5*\x,.5)--(.5*\x,0);
}
\def\nn{1};
\foreach \x in {1,...,\nn}{
\draw (.5*\x,-.5)--(.5*\x+.5,-.5)--(.5*\x+.5,0)--(.5*\x,0)--(.5*\x,-.5);
}
\def\nnn{1};
\foreach \x in {1,...,\nnn}{
\draw (.5*\x,-1)--(.5*\x+.5,-1)--(.5*\x+.5,-.5)--(.5*\x,-.5)--(.5*\x,-1);
}
\end{tikzpicture} &&
\begin{tikzpicture}
\def\n{3};
\foreach \x in {1,...,\n}{
\draw (.5*\x,0)--(.5*\x+.5,0)--(.5*\x+.5,.5)--(.5*\x,.5)--(.5*\x,0);
}
\def\nn{2};
\foreach \x in {1,...,\nn}{
\draw (.5*\x,-.5)--(.5*\x+.5,-.5)--(.5*\x+.5,0)--(.5*\x,0)--(.5*\x,-.5);
}
\end{tikzpicture} \\ &&&&& \\
$\chi(V_{(3,3,1)})$ & = && $\chi(V_{(3,2,1)})$ & + & $\chi(V_{(3,3)})$ \\ &&&&& \\
&& $-$ & $\chi(V_{(3,1,1)})$ & $-$ & $\chi(V_{(3,2)})$ \\
\end{tabular}
\end{center}
\caption{Decomposition of $V_{(3,3,1)}$}
\label{fig:2}       
\end{figure}

\section{Main Results}
\label{res}

In \cite{B}, Bertrand works only with Young diagrams which have at most two rows, where the decomposition detailed above involves at most five resulting terms. This allows for  a very explicit recurrence to be derived from the decomposition. Because this is not feasible for larger diagrams, we take a slightly different approach. Thm. \ref{thm1} tells us that after fixing the number of boxes in the second row of a Young diagram, as we add boxes to the first row of the diagram, we cycle through one of three different polynomials, alternating signs. Here, we opt to use six different polynomials in order to omit the sign. As in Thm. \ref{thm1}, we expect that after fixing some portion of the diagram and subsequently adding boxes to the remainder, we should use the same polynomial for diagrams whose number of added boxes differs by a multiple of six. We find that, by analogue with \cite{B}, we may add boxes one at a time to the first row, requiring all other rows to be fixed. We proceed, then, by establishing a recurrence between diagrams whose first rows differ by exactly six boxes, and which are otherwise identical; this is the bulk of the proof of the resulting theorem. Subsequent analysis of the resulting recurrence yields the desired result, entirely analogous to Thm. \ref{thm1}.

\setcounter{thm}{1}

\begin{thm}
   Let $\lambda = (\lambda_1, \lambda_2, ..., \lambda_k)$ be a Young diagram of size $n$ where $\lambda_1 - \lambda_2 \geq 6$. Then there exist six polynomials, $p_0, p_1, p_2, p_3, p_4$, and $p_5$, each of degree at most $n - \lambda_1 = \sum_{i=2}^k \lambda_i$, such that \\ $\chi(H_C^*(S_n, V_{\lambda})) = p_i(n)$ where $n \equiv i \mod 6$.
\end{thm}

\begin{proof}

We proceed by inducting on $n-\lambda_1$, the number of boxes in all but the first row of $\lambda$. As shown in \cite{B}, and more explicitly in \cite{LL}, the statement is true when $n - \lambda_1=0$.

Given a diagram $\lambda = (\lambda_1, \lambda_2, ..., \lambda_k)$, we suppose that the theorem is true for all diagrams $\mu = (\mu_1, \mu_2, ..., \mu_l)$ for which the sum $\sum_{j=2}^l \mu_j$ is less than $n - \lambda_1$. We begin by establishing a recurrence relation which relates the values of $\chi(V_{(\lambda_1, ..., \lambda_k)})$ and $\chi(V_{(\lambda_1-6, ..., \lambda_k)})$, which we will obtain by repeatedly applying our decomposition. 

The diagrams corresponding to the terms which result from the decomposition will have fewer boxes than the original, and so the inductive hypothesis will apply to most of the resulting terms, allowing us to replace each of these terms with a polynomial. A priori, these may not be polynomials in $n$: the term corresponding to a diagram with one box fewer than the original would, for example, be a polynomial in $n-1$. However, algebraic simplification allows us to write each such term, and also their sum, as a single polynomial in $n$. By our hypothesis, the highest degree of any resulting polynomial is at most $n - \lambda_1 - 1$, and so their sum also has degree at most $n - \lambda_1 - 1$.

The only terms for which the inductive hypothesis does not apply are those for which all boxes are removed from the first row. To simplify our calculations, we will therefore focus only on these exceptions, each of the form $V_{(*, \lambda_2, ..., \lambda_k)}$, and assume that all remaining terms comprise a single additional polynomial in $n$, which we denote by $\varepsilon$ (note that $\varepsilon$ varies line to line). We apply the decomposition repeatedly, as follows:

\begin{center}
    \begin{tabular}{lcl}
        $\chi(V_{(\lambda_1, ..., \lambda_k)})$ & = & $\chi(V_{(\lambda_1-1, ..., \lambda_k)}) - \chi(V_{(\lambda_1-2, ..., \lambda_k)}) + \varepsilon$ \\
        & = & $\chi(V_{(\lambda_1-2, ..., \lambda_k)}) - \chi(V_{(\lambda_1-3, ..., \lambda_k)}) -\chi(V_{(\lambda_1-2, ..., \lambda_k)}) + \varepsilon$ \\
         & = & $-\chi(V_{(\lambda_1-3, ..., \lambda_k)}) + \varepsilon$ \\
         & = & $-\chi(V_{(\lambda_1-4, ..., \lambda_k)}) + \chi(V_{(\lambda_1-5, ..., \lambda_k)}) + \varepsilon$ \\
         & = & $-\chi(V_{(\lambda_1-5, ..., \lambda_k)}) + \chi(V_{(\lambda_1-6, ..., \lambda_k)}) + \chi(V_{(\lambda_1-5, ..., \lambda_k)}) + \varepsilon$ \\
         & = & $\chi(V_{(\lambda_1-6, ..., \lambda_k)}) + \varepsilon$ \\
    \end{tabular}
\end{center}

Since a given $p_i$ is used only for values of $n$ that are equivalent to $i$ mod 6, we define $m$ by $n=6m+i$; this way, values of $n$ which differ by six correspond to consecutive values of $m$. Therefore, $\chi(V_{(\lambda_1, ..., \lambda_k)})$ and $\chi(V_{(\lambda_1-6, ..., \lambda_k)})$ are consecutive with respect to $m$, so we let \\ $a_m = \chi(V_{(\lambda_1, ..., \lambda_k)})$. From above, $\chi(V_{(\lambda_1, ..., \lambda_k)})$ is equal to $\chi(V_{(\lambda_1-6, ..., \lambda_k)})$ plus some polynomial in $n$: call this polynomial $q(n)$. We may then write the equation above as the recurrence $a_m = a_{m-1} + q(n)$; by applying the change of coordinates $n=6m+i$, we rewrite $q(n)$ as a polynomial $\hat{q}$ in $m$. This recurrence is then equivalent to 

\[a_m - a_{m-1} = \hat{q}(m). \]

Let the closed form for the $a_i$ be given by the function $f(i)$. Note that the left hand side of the equation above is then equal to \\ $f(m) - f(m-1) = \Delta(f)$, the discrete derivative of $f$. This then gives

\[ \Delta f = \hat{q}(m).\]

By the power rule for discrete calculus (see, for example, \cite{CHR}), we have

\[ f = \hat{Q}(m) \]

for some polynomial $\hat{Q}$. Further, since $q$ (and also $\hat{q}$) had degree at most $n-\lambda_1 -1$, $\hat{Q}$ has degree at most $n  - \lambda_1$. By again applying a change of coordinates, we may rewrite $\hat{Q}$ as a polynomial $Q$ in $n$, which also has degree at most $n - \lambda_1$.

Thus, for given values of $\lambda_1, ..., \lambda_k$ and a given value of $n \mod 6$, $\chi(V_{(\lambda_1, ..., \lambda_k)})$ is given by a polynomial in $n$ of degree at most $n- \lambda_1$. It follows that there may be six distinct such polynomials, one for each value of $n \mod 6$.

\end{proof}

\bibliographystyle{plain}
\bibliography{bibtex}

\end{document}